\documentclass{amsart}
\usepackage{graphicx, amssymb, eucal, hyperref}
\newtheorem{thm}{Theorem}[section]

\newtheorem{lem}[thm]{Lemma}
\newtheorem{fact}[thm]{Fact}
\newtheorem{prop}[thm]{Proposition}
\theoremstyle{definition}
\newtheorem{problem}[thm]{Problem}
\newtheorem{defn}[thm]{Definition}
\theoremstyle{remark}
\newtheorem{rmk}[thm]{Remark}

\numberwithin{equation}{section}
\newcommand{\vertiii}[1]{{\left\vert\kern-0.25ex\left\vert\kern-0.25ex\left\vert #1 
    \right\vert\kern-0.25ex\right\vert\kern-0.25ex\right\vert}}
\newcommand{\norm}[1]{\left\Vert#1\right\Vert}

\newcommand{\set}[1]{\left\{\,#1\,\right\}}
\newcommand{\defined}[1]{\emph{#1}}
\newcommand{\mc}[1]{\mathcal{#1}}
\newcommand{\op}[1]{\operatorname{#1}}
\newcommand{\ignore}[1]{}

\renewcommand{\phi}{\varphi}

\newcommand{\cO}{\mathcal O}
\newcommand{\cZ}{\mathcal Z} 
\newcommand{\cR}{\mathcal R} 
\newcommand{\bbC}{\mathbb C} 
 
\newcommand{\bbN}{\mathbb N} 
\newcommand{\cW}{\mathcal W} 
\DeclareMathOperator{\Ad}{Ad}
\DeclareMathOperator{\SU}{SU}

\DeclareMathOperator{\Aut}{Aut}
\newcommand{\cK}{\mathcal{K}}
\begin{document}

\title[Fra\"iss\'e limits of C*-algebras]{Fra\"iss\'e limits of C*-algebras}%
\author[C.J. Eagle]{Christopher J. Eagle}
 
\address{Department of Mathematics\\ 
University of Toronto \\
40 St. George Street \\
Toronto, Ontario, Canada\\ 
M5S 2E4}
\email{eaglec@uvic.ca}
\curraddr{
Department of Mathematics and Statistics \\
University of Victoria\\
P.O. Box 1700 STN CSC \\
Victoria, British Columbia\\
Canada, V8W 2Y2
}
\urladdr{http://www.math.toronto.edu/cjeagle/}

\author[I. Farah]{Ilijas Farah}

\address{Department of Mathematics and Statistics\\
York University\\
4700 Keele Street\\
North York, Ontario\\ Canada, M3J
1P3\\
and Matematicki Institut, Kneza Mihaila 34, Belgrade, Serbia}

\email{ifarah@mathstat.yorku.ca}
\urladdr{http://www.math.yorku.ca/$\sim$ifarah}

\author[B. Hart]{Bradd Hart}
\address{Department of Mathematics and Statistics\\
McMaster University\\ 1280 Main Street\\ West Hamilton, Ontario\\
Canada L8S 4K1}
\email{hartb@mcmaster.ca}

\urladdr{http://www.math.mcmaster.ca/$\sim$bradd/}

  \author[B. Kadets]{Boris Kadets}
\address{School of Mathematics and Mechanical Engineering\\
   Kharkiv V.N. Karazin National University\\
   Kharkiv, Ukraine}
   \email{borja.kadets@gmail.com}
   
   \author[V. Kalashyk]{Vladyslav Kalashnyk}
   \address{Department of Mathematics\\
   340 Rowland Hall\\
University of California, Irvine\\
Irvine, CA 92697-3875}
   \email{vkalashn@uci.edu}

    \author[M. Lupini]{Martino Lupini}
    \address{Department of Mathematics and Statistics\\
York University\\
4700 Keele Street\\
North York, Ontario\\ Canada, M3J
1P3}
\email{lupini@caltech.edu}
\curraddr{Mathematics Department\\
California Institute of Technology\\
1200 E. California Blvd\\
MC 253-37 \\
Pasadena, CA 91125}
\urladdr{http://www.lupini.org}

\thanks{The genesis of this paper was a Fields Institute summer research project.  The authors would like to thank the Institute for an excellent research environment. We would also like to thank Aaron Tikuisis for helpful conversations and Shuhei Masumoto for several corrections.  Finally, we thank the referee for comments which led to numerous clarifications and improvements.}

\date{\today}%
\begin{abstract}
We realize the Jiang-Su algebra, all UHF
algebras, and the hyperfinite II$_{1}$ factor as Fra\"iss\'e limits of suitable
classes of structures. Moreover by means of Fra\"iss\'e theory we provide new
examples of AF algebras with strong homogeneity properties. As a consequence
of our analysis we deduce Ramsey-theoretic results about the class of
full-matrix algebras.   
\end{abstract}
\maketitle

\section{Introduction}
Fra\"iss\'e theory lies at the crossroads of
combinatorics and model theory. It originates from the seminal work of Fra\"iss\'e in \cite%
{fraisse_sur_1954} for the case of discrete countable structures. Broadly
speaking, Fra\"iss\'e theory studies the correspondence between homogeneous
structures and properties of the classes of their finitely generated
substructures. The \emph{age} of a countable structure is the collection of its
finitely generated substructures, and the ages of homogeneous structures are
precisely the classes of structures known as \emph{Fra\"iss\'e classes}.  Conversely, given any Fra\"iss\'e
class one can construct a countable homogeneous structure
with the given class as its age.  This structure, which is referred to as the \emph{Fra\"iss\'e limit} of the class, is unique up to isomorphism, and can be thought of as the structure generically constructed from the class.

Fra\"iss\'e theory has been recently generalized to  metric
structures by Ben Yaacov in \cite{BenYaacov2013}.  
   An earlier approach to Fra\"iss\'e limits in the metric setting was developed in~\cite{Schoretsanitis2007}. 
 Standard examples of metric Fra\"iss\'e limits are the Urysohn metric space, its variants, 
 and the Gurarij Banach space (previously construed as a Fra\"iss\'e limit in~\cite{kubis_proof_2013}). 
 The  Elliott intertwining argument central in classification program for nuclear C*-algebras (see \cite{rordam_classification_2002})
 is closely related to the proof of uniqueness of metric Fra\"iss\'e limits.

In this paper we study Fra\"iss\'e limits of C*-algebras. In particular we show
that several important C*-algebras can be described as Fra\"iss\'e limits of suitable
classes. As in \cite{Melleray2014}, we work under slightly less general
assumptions than \cite{BenYaacov2013}, and we consider only classes where the
interpretation of functional and relational symbols are Lipschitz (see Section \ref{Section:FraisseLimit} below for the precise definitions).  
In our constructions we consider Fra\"iss\'e classes that are not complete (in the sense of \cite{BenYaacov2013}) and
are not closed under substructures. The reason we do this is that the class of finitely generated
substructures of a given C*-algebra tends to be too large. 
As a matter of fact, conjecturally all simple and separable C*-algebras are singly generated 
(see \cite{Hannes2014}). 
 As a consequence we
only consider classes that are made of suitable \textquotedblleft
small\textquotedblright\ subalgebras of the given C*-algebra.

We show that the Jiang-Su algebra $\mathcal{Z}$ \cite{JiangSu} and 
all UHF algebras \cite{glimm_certain_1960} are limits of suitable
Fra\"iss\'e classes. Both $\mathcal{Z}$ and UHF algebras are examples of
C*-algebras of fundamental importance for the classification program of
C*-algebras, a survey of which can be found in 
\cite{rordam_classification_2002}, \cite{elliott_regularity_2008}. Furthermore we prove that, while the class of
finite-dimensional C*-algebras is not Fra\"iss\'e, one can obtain a
Fra\"iss\'e class by adding a distinguished interior trace and imposing a
restriction on the number of direct summands. This provides new 
examples of AF algebras satisfying strong homogeneity properties. Finally
we deduce a Ramsey-type result for the class of matrix algebras, either endowed with
 the operator norm or with the trace-norm. This is obtained from the 
 above mentioned description of (infinite type) UHF algebras as limits, together
with a similar characterization of the hyperfinite II$_{1}$ factor $\mathcal{%
R}$. We use the observation that the corresponding automorphism groups are extremely
amenable which is a result due to Gromov \cite{gromov_filling_1983}. The other ingredient is the well known connection between extreme amenability and Ramsey-theoretic
properties of a Fra\"iss\'e class originally established in \cite%
{kechris_fraisse_2005} and recently generalized to the metric setting in 
\cite{Melleray2014}.

The paper is divided into seven sections. In Section \ref{Section:FraisseLimit} we recall the basic notions and results of Fra\"iss\'e
theory, adapted to the framework of C*-algebras. Section~\ref{Section:AF}
contains the results about UHF algebras, AF algebras, and the hyperfinite II$_1$ factor. The description of
the Jiang-Su algebra as a Fra\"iss\'e limit is presented in Section~\ref{Section:JiangSu}. We recall the notions of L\'{e}vy groups
and extremely amenable groups in Section \ref{Section:Levy}, where we observe
that the automorphisms groups of the hyperfinite~II$_1$ factor and infinite type UHF
algebras are L\'{e}vy. This is used in Section \ref{Section:Ramsey} to deduce
Ramsey-type results about the class of full matrix algebras. We conclude in
Section~\ref{Section:future} with a discussion of future lines of research
and open problems.

\section{Fra\"iss\'e limits of C*-algebras}\label{Section:FraisseLimit}
In this section we define Fra\"iss\'e classes of C*-algebras and their Fra\"iss\'e limits.  Recall that a \emph{C*-algebra} is a subalgebra of the algebra of bounded linear operators on a Hilbert space which is closed under the adjoint operation * and is closed in the operator norm topology (see \cite{Davidson1996} for an introduction to C*-algebras).  We will often consider \emph{unital} C*-algebras, that is, algebras with a multiplicative identity element, but when we say ``C*-algebra" without qualification we mean an algebra which is not necessarily unital.  We will consider C*-algebras as examples of metric structures.  The literature contains several definitions of metric structures suited to various purposes; the one we present here is the same as in \cite{Melleray2014}.

\begin{defn}
A \defined{language} $L$ consists of a set of \emph{predicate symbols} and a set  of \emph{function symbols}.  Each predicate symbol $P$ and function symbol $f$ carries an associated \emph{arity} and \emph{Lipschitz constant} $C_P$ and $C_f$ respectively.  We assume that every language includes a distinguished symbol $d$, which will always be interpreted as a metric.

An \emph{$L$-structure} is a complete metric space $(A, d)$, together with interpretations for the symbols of $L$: for each
\begin{enumerate}
\item $n$-ary  predicate symbol $P$, a $C_P$-Lipschitz function $P^A : A^n \to \mathbb{R}$, and
\item each $n$-ary function symbol $f$, a $C_f$-Lipschitz function $f^A : A^n \to A$.
\end{enumerate}
\end{defn}

We need to say a word about how we will formally see C*-algebras as structures in the sense of the previous definition.  For a C*-algebra $A$, we will consider the unit ball $A_1$ together with the operator norm as the underlying complete metric space.  In terms of the language, for every *-polynomial $p(x_1,\ldots,x_n)$, there will be an $n$-ary predicate $R_p$ which is interpreted on $A_1^n$ by $\|p(x_1,\ldots,x_n)\|$.  This relation is Lipschitz with a constant that is independent of the choice of C*-algebra.  If we wish to consider a trace as well then we similarly introduce relations for traces of all *-polynomials on the unit ball; again, all of these relations are Lipschitz. In practice, we will use the usual C*-algebraic notation when we deal with C*-algebras but formally, for the purposes of fitting the continuous Fra\"iss\'e context, we will treat them as above.


Since all of our structures fit into the framework described above, we find it convenient to give a presentation of Fra\"iss\'e theory which is closer to that of \cite{Melleray2014} than the more general approach taken in \cite{BenYaacov2013}.  Our definitions are not identical to those of either \cite{Melleray2014} or \cite{BenYaacov2013}; see Remark \ref{rmk:DifferencesInDefinition} for discussion of the differences. In particular, \cite{Melleray2014} do not require their metric spaces to be bounded and therefore 
their structures are not metric structures in the sense of \cite{BYBHU}. 
In the way that we are viewing C*-algebras (as their unit balls) the underlying metric is bounded and therefore their continuous theory is well-defined.

\begin{defn}
Let $A$ be a C*-algebra, and $\overline{a}$ a tuple from $A_1$.  The \defined{subalgebra generated by $\overline{a}$} is the smallest C*-subalgebra of $A$ which contains $\overline{a}$, and is denoted~$\langle \overline{a} \rangle$.  We say $A$ is \defined{finitely generated} if there is a finite tuple $\overline{a}$ such that $A = \langle \overline{a} \rangle$.
\end{defn}

\begin{rmk}
The condition that a C*-algebra be finitely generated may be weaker than it appears.  It is known that a large class of separable unital C*-algebras, including all those which are $\mc{Z}$-stable, are generated by single elements (see \cite{Hannes2014} for this result and further discussion).  In particular, some of the C*-algebras we will construct as 
Fra\"iss\'e limits will be singly generated.  
\end{rmk}

\begin{defn}\label{Definition:Fraisse}
Let $\mc{K}$ be a class of finitely generated structures with distinguished generators.
\begin{enumerate}
\item{
We say that a structure is a \defined{$\mc{K}$-structure} if it is an inductive limit of elements of $\mc{K}$.
}
\item{
The class $\mc{K}$ has the \defined{near amalgamation property (NAP)} if whenever\\ $A, B_0, B_1 \in \mc{K}$, and $\phi_i : A \to B_i$ are morphisms, then for every $\epsilon > 0$ there is a $C\in \mc{K}$ and morphisms $\psi_i : B_i \to C$ such that $d(\psi_0\phi_0(\overline{a}), \psi_1\phi_1(\overline{a})) < \epsilon$, where $\overline{a}$ is the distinguished generating set of $A$.
}
\item{
The class $\mc{K}$ has the \defined{amalgamation property (AP)} if, in the definition of NAP, we may take $\epsilon = 0$.
}
\item{
The class $\mc{K}$ has the \defined{joint embedding property (JEP)} if for all $A, B \in \mc{K}$ there is $C \in \mc{K}$ such that $A$ and $B$ embed into $C$.
}
\end{enumerate}
\end{defn}

The properties defined above have clear analogues in classical Fra\"iss\'e theory.  
In the classical setting one works with \emph{countable} classes of finite structures, in order to ensure that the resulting limit object is also countable.  In the metric setting it is necessary to replace countability by separability in a suitably chosen topology, which we now describe.  As in \cite[Definition 2.10]{BenYaacov2013}, if $\mc{K}$ is a class of finitely generated structures, we denote by $\mc{K}_n$ the subclass of $\mc{K}$ consisting of all members of $\mc{K}$ whose distinguished generating sets have size $n$.  If $\mc{K}$ has JEP and NAP, we can define a pseudo-metric on $\mc{K}_n$ by defining
\[d^{\mc{K}}(\overline{a}, \overline{b}) = \inf\{d_C(\overline{a}, \overline{b}) : \overline{a}, \overline{b} \in C,  C \in \mc{K}\}\]
where $d_C$ is the distance computed in $C$ (see \cite[Definition 2.11]{BenYaacov2013}) and $\bar a$ and $\bar b$ are the distinguished generators of elements of $\mc{K}_n$.

\begin{defn}
A class $\mc{K}$ of finitely generated structures with JEP and NAP has the \defined{weak Polish Property (WPP)} if for each $n$ the pseudo-metric space $(\mc{K}_n, d^{\mc{K}})$ is separable.
\end{defn}

Finally, we come to the central definitions of Fra\"iss\'e classes and Fra\"iss\'e limits.

\begin{defn}
A class $\mc{K}$ of finitely generated structures is a \defined{Fra\"iss\'e class} if it satisfies JEP, NAP and WPP.

A $\mc{K}$-structure $M$ is a \defined{Fra\"iss\'e limit} of the Fra\"iss\'e class $\mc{K}$ if:
\begin{enumerate}
\item{
$M$ is \defined{$\mc{K}$-universal}: For every $A \in \mc{K}$ there is an embedding of $A$ into $M$,
}
\item{
$M$ is \defined{approximately $\mc{K}$-homogeneous}: for all $A, B \subseteq M$ such that $A \cong B$, $A, B \in \mc{K}$ and for every $\epsilon > 0$ there is an automorphism $\sigma$ of $M$ such that if $\bar a$ and $\bar b$ are the generators of $A$ and $B$ then $d(\bar a,\sigma(\bar b)) < \epsilon$.
}
\end{enumerate}
\end{defn}

\begin{rmk}\label{rmk:DifferencesInDefinition}
The classes that we are considering are \emph{incomplete} in the sense of \cite[Definition 2.12]{BenYaacov2013}.  The completions of our classes will include their Fra\"iss\'e limits.  The classes we consider also fail to be \emph{hereditary}, that is, we will have classes $\mc{K}$, and members $A \in \mc{K}$ with finitely generated substructures $B \subseteq A$ and $B \not\in \mc{K}$.  As a consequence, we do not have the usual correspondence between Fra\"iss\'e classes and ages of homogeneous structures.  Nevertheless, our definitions do allow us to construct limits of Fra\"iss\'e classes, and hence obtain interesting information about the limit objects. 
\end{rmk}

\begin{thm}\label{lem:FraisseLimitsExist}
Every Fra\"iss\'e class has a Fra\"iss\'e limit which is unique up to isomorphism.
\end{thm}


The proof is a straightforward adaptation of the proofs of Lemma 2.17 and Theorem 2.19 from \cite{BenYaacov2013}.

In the discrete setting many (though not all) well-known Fra\"iss\'e limits have theories with quantifier elimination.  The main results of \cite{EagleFarahKirchbergVignati} show that quantifier elimination is a rare phenomenon for C*-algebras; in particular, it is shown in \cite{EagleFarahKirchbergVignati} that the only noncommutative C*-algebra with quantifier elimination is $M_2(\mathbb{C})$, so none of the noncommutative C*-algebras we construct as Fra\"iss\'e limits in the subsequent sections have quantifier elimination.  In Section 3.2 we show that the hyperfinite II$_{1}$ factor $\mc{R}$ is the Fra\"iss\'e limit of matrix algebras viewed as von Neumann algebras.  The theory of $\mc{R}$ also does not have quantifier elimination, as shown in \cite{HartGoldbringSinclair}.

We do have one example of a C*-algebra which is a Fra\"iss\'e limit whose theory has quantifier elimination, namely the algebra $C(2^{\mathbb{N}})$ of continuous functions on the Cantor set.  It is straightforward to see that this algebra is the Fra\"iss\'e limit of the class of finite-dimensional commutative C*-algebras (i.e., the algebras of the form $\mathbb{C}^n$).  Quantifier elimination for the theory of $C(2^{\mathbb{N}})$ is proved in \cite{EagleVignati}.  In fact, by the results of \cite{EagleGoldbringVignati}, the theory of $C(2^{\mathbb{N}})$ is the only theory of infinite-dimensional commutative unital C*-algebras which has quantifier elimination.

\section{AF algebras}\label{Section:AF}
We now turn to describing several examples of Fra\"iss\'e classes of finite-dimensional C*-algebras.  Throughout this section, when we discuss $M_n(\mathbb{C})$ we are considering it as being $n^2$-generated by the standard matrix units.  Recall that a \emph{(normalized) trace} on a unital C*-algebra $A$ is a continuous linear functional $\tau : A \to \mathbb{C}$ such that $\tau(1) = 1$, it is positive (i.e., $\tau(a^*a) \geq 0$ for all $a \in A$), and $\tau(ab) = \tau(ba)$ for all $a, b \in A$.  The space of traces of $A$, $T(A)$, is a weak*-compact and convex
subset of  the unit ball of the dual of $A$. Every unital *-homomorphism between tracial algebras 
$\phi\colon A\to B$ gives rise to the 
continuous affine map $\phi_*\colon T(B)\to T(A)$:  if $\tau \in T(B)$ and $a \in A$, define
\[
\phi_*(\tau)(a)=\tau(\phi(a)). 
\]
This contravariant functor will also play a role  in the proof of Lemma~\ref{lem:JiangSuProperties}.
It is a well-known fact from linear algebra that each matrix algebra $M_n(\mathbb{C})$ has a unique trace, and that trace $\tau$ is given by $\tau([a_{i, j}]) = \frac{1}{n}\sum_{j=1}^n a_{j, j}$. We will make frequent and unmentioned use of the following well-known properties of finite-dimensional C*-algebras. 
\begin{fact}
\begin{enumerate}
\item{
Every finite-dimensional C*-algebra is isomorphic to a finite direct sum of matrix algebras.
}
\item{
If $A = M_{k_1}(\mathbb{C}) \oplus \cdots \oplus M_{k_n}(\mathbb{C})$, then every trace on $A$ is a convex combination of the (unique) traces on $M_{k_1}(\mathbb{C}), \ldots, M_{k_n}(\mathbb{C})$.
}
\item{
There is a unital embedding of $M_n(\mathbb{C})$ into $M_m(\mathbb{C})$ if and only if $n$ divides $m$.  A unital embedding of finite-dimensional algebras $A$ and $B$ is characterized up to unitary conjugacy by the multiplicities with which it maps each direct summand of $A$ into each direct summand of $B$ (that is, by its Bratteli diagram; see \cite[Section III.2]{Davidson1996} or \cite[Section 4.4]{Farah2013a}).
}
\end{enumerate}
\end{fact}

When we consider a finite-dimensional algebra $M_{k_1}(\mathbb{C}) \oplus \cdots \oplus M_{k_n}(\mathbb{C})$, we always consider it as being generated by elements of the form $a_1 \oplus \cdots \oplus a_n$, where the $a_i$'s vary over the distinguished generators of the $M_{k_i}(\mathbb{C})$'s.

We begin by observing that when we consider classes of finite-dimensional C*-algebras near amalgamation can be replaced by actual amalgamation.

\begin{lem}
Let $\mc{K}$ be a subclass of the class of finite-dimensional C*-algebras.  The following are equivalent:
\begin{enumerate}
\item{
$\mc{K}$ has NAP,
}
\item{
$\mc{K}$ has AP.
}
\end{enumerate}
\end{lem}
\begin{proof}
The direction (2) $\implies$ (1) is obvious.  For the other direction, suppose that $\mc{K}$ has NAP.  Take $A, B_1, B_2 \in \mc{K}$, and let $\phi_i : A \to B_i$ be morphisms.  Write $A = M_{n_1}(\mathbb{C}) \oplus \cdots \oplus M_{n_k}(\mathbb{C})$.  Let 
$$\overline{a} = ((I_{n_1}, 0, \ldots, 0), (0, I_{n_2}, 0, \ldots, 0), \ldots, (0, \ldots, 0, I_{n_k}))$$  By definition of NAP, with $\epsilon = \frac{1}{2}$, there is a $\mc{K}$-structure $C$, and maps $\psi_i : B_i \to C$ such that $d(\psi_0\phi_0(\overline{a}), \psi_1\phi_1(\overline{a})) < \epsilon$.  

We claim that $C$ satisfies the definition of AP.  Consider the Bratteli diagrams of the embeddings $\psi_0\phi_0$ and $\psi_1\phi_1$ of $A$ into $C$.  If these Bratteli diagrams are the same, then after conjugating by a unitary we have $\psi_0\phi_0 = \psi_1\phi_1$, so $C$ exactly amalgamates $B_0$ and $B_1$ over $A$.  If the Bratteli diagrams are not the same, then for some $i$ the ranks of the matrices $\psi_0\phi_0(0, \ldots, I_{n_i}, \ldots, 0))$ and $\psi_1\phi_1(0, \ldots, I_{n_i}, \ldots, 0)$ are not equal.  These images are then projections of different ranks, so
\[\norm{\psi_0\phi_0(0, \ldots, I_{n_i}, \ldots, 0)) - \psi_1\phi_1(0, \ldots, I_{n_i}, \ldots, 0)} = 1,\]
which contradicts our choice of $C$.
\end{proof}

In the setting of Banach spaces, the class of all finite-dimensional Banach spaces is a Fra\"iss\'e class, with the Gurarij space as its limit (see \cite[Section 3.3]{BenYaacov2013}).  By contrast, the class of all finite-dimensional C*-algebras is \emph{not} a Fra\"iss\'e class.  The obstacle to amalgamation comes from considering traces.

\begin{prop}\label{prop:AllFDNotFraisse}
The class of finite-dimensional C*-algebras is not a Fra\"iss\'e class.
\end{prop}
\begin{proof}
We show that this class does not have AP.  Let $A = \mathbb{C} \oplus \mathbb{C}$, $B = M_2(\mathbb{C})$, and $C = M_3(\mathbb{C})$.  Consider the following embeddings $\iota_{A, C} : A \to C$ and $\iota_{B, C} : B \to C$:
\[\iota_{A, C}(a, b) = \begin{bmatrix}a & 0 \\ 0 & b\end{bmatrix} \qquad \qquad \iota_{B, C}(a, b) = \begin{bmatrix}a & 0 & 0 \\ 0 & b & 0 \\ 0 & 0 & b\end{bmatrix}.\]  
Suppose that $D$ is a finite-dimensional C*-algebra which amalgamates $B$ and $C$ over $A$ with respect to these embeddings, via embeddings $\iota_{B, D}$ and $\iota_{C, D}$.  Let $x = \iota_{B, D} \circ \iota_{A, B}(1, 0)$, and note that $x = \iota_{C, D}\circ \iota_{A, C}(1, 0)$ by definition of amalgamation.  

Let $\tau_D$ be a trace on $D$.  On the image of $B$ in $D$ the trace $\tau_D$ restricts to a trace, which must be the unique trace $\tau_B$ from $B$.  Therefore,
\[\tau_D(x) = \tau_B(\iota_{A, B}(1, 0)) = \tau_B\left(\begin{bmatrix}1 & 0 \\ 0 & 0\end{bmatrix}\right) = \frac{1}{2}.\]
Similarly, $\tau_D$ restricts to the unique trace $\tau_C$ on the image of $C$ in $D$.  Then we have
\[\tau_D(x) = \tau_C(\iota_{A, C}(1, 0)) = \tau_C\left(\begin{bmatrix}1 & 0 & 0 \\ 0 & 0 & 0 \\ 0 & 0 & 0 \end{bmatrix}\right) = \frac{1}{3}.\]
This contradiction finishes the proof.  
\end{proof}

\subsection{UHF algebras}
If we restrict our attention to subclasses of the class of matrix algebras, we can obtain UHF algebras as Fra\"iss\'e limits.  Recall that a separable unital C*-algebra which arises as the direct limit of unital embeddings of matrix algebras is called a \defined{uniformly hyperfinite (UHF)} algebra.  It is well-known that UHF algebras are classified by \emph{supernatural numbers}, that is, formal products $\prod_{p \text{ prime}}p^{n_p}$, where each $n_p \in \mathbb{N} \cup \set{\infty}$; given a UHF algebra $A$, which is the direct limit of $M_{k_1}(\mathbb{C})\to M_{k_2}(\mathbb{C}) \to \cdots$, the associated supernatural number is given by $n_p = \sup\set{r : p^r \mid k_i \text{ for some $i$}}$.  See \cite[Chapter 4]{Farah2013a} for more details.

\begin{thm}\label{thm:UHF}
Every UHF algebra is a Fra\"iss\'e limit.
\end{thm}
\begin{proof}
Let $A$ be a UHF algebra, and write $A$ as the direct limit of matrix algebras $M_{n_1}(\mathbb{C}), M_{n_2}(\mathbb{C}), \ldots$.  As usual, we view each $M_{n_i}(\mathbb{C})$ with its standard matrix units as generators.  Let $\mc{K} = \set{M_{n_i}(\mathbb{C}) : i \in \mathbb{N}}$.  We then have
\[\mc{K}_n = \begin{cases}\{M_{n_i}(\mathbb{C})\} &\text{ if $n = n_i^2$ for some $i$} \\ \emptyset &\text{ if $n \neq n_i^2$ for all $i$}\end{cases}.\]
In particular, it is clear that WPP holds.  The class $\mc{K}$ has a minimal element $M_{n_1}(\mathbb{C})$, so JEP will be a consequence of AP.  To see AP, note that if $M_{n_i}(\mathbb{C})$ is embedded in $M_{n_j}(\mathbb{C})$ and $M_{n_k}(\mathbb{C})$, and (without loss of generality) $n_j \leq n_k$, then $M_{n_j}(\mathbb{C})$ embeds in $M_{n_k}(\mathbb{C})$ in a way which (up to unitary equivalence) respects the embedding of $M_{n_i}(\mathbb{C})$.  Therefore $M_{n_k}(\mathbb{C})$ itself serves to amalgamate $M_{n_j}(\mathbb{C})$ and $M_{n_k}(\mathbb{C})$ over $M_{n_i}(\mathbb{C})$.

It is clear from the construction of the Fra\"iss\'e limit of $\mc{K}$ that this limit is a UHF algebra with the same supernatural number as $A$, and hence is isomorphic to~$A$.
\end{proof}

An argument similar to the one in Theorem \ref{thm:UHF} shows that the class of full matrix algebras with injective (not necessarily unital) *-homomorphisms as morphisms is a Fra\"iss\'e class. The corresponding limit is the (non-unital) C*-algebra  of compact operators on the separable infinite-dimensional Hilbert space \cite[Section I.8]{blackadar_operator_2006}.

\subsection{The hyperfinite II$_1$ factor}
In Theorem \ref{thm:UHF} matrix algebras were regarded as finite-dimensional C*-algebras, but we can also regard them as tracial von Neumann algebras. A \emph{tracial von Neumann} algebra is a unital C*-algebra $M$ endowed with a distinguished trace $\tau$ such that the unit
ball of $M$ is complete with respect to the trace-norm 
$\|x\| _{\tau }=\tau \left( x^{\ast }x\right) ^{\frac{1}{2}}$. As was shown in \cite{Farah2014a}, tracial
von Neumann algebras can be regarded as metric structures in the language of
unital C*-algebras with  the additional predicate symbol for a distinguished trace, where the symbol for the metric is interpreted as the distance associated with the trace-norm. Since the operator norm is not uniformly continuous with respect to the trace-norm, it is 
no longer  part of the structure. A
tracial von Neumann algebra is \emph{separable }if it is separable with
respect to the trace-norm.

A \emph{finite factor }is a tracial von Neumann algebra $\left( M,\tau
\right) $ such that the center of $M$ consists only of the scalar multiples
of the identity. When $M$ is a finite factor the trace $\tau $ on $M$ is uniquely
determined. Full matrix algebras are examples of finite factors. A finite
factor that is not isomorphic to a full matrix algebra is called a II$_{1}$
factor. Equivalently, a finite factor is a II$_{1}$ factor when the trace
assumes all values in $[0, 1]$ on projections.

A II$_{1}$ factor is \emph{hyperfinite }if it can be locally approximated
(in trace-norm) by full matrix algebras. 
The unique hyperfinite II$_{1}$ factor is traditionally denoted by $\mathcal{%
R}$, and can be concretely realized as the direct limit of the
direct sequence $\left( M_{2^{n}}(\mathbb{C})\right) _{n\in \mathbb{N}}$ in
the category of tracial von Neumann algebras. The same proof as Theorem \ref%
{thm:UHF} shows that the class of full matrix algebras regarded as finite
factors is a Fra\"iss\'e class. Since a direct limit of finite factors is a
finite factor, the Fra\"iss\'e limit of the class of full matrix algebras is the
hyperfinite II$_{1}$ factor $\mathcal{R}$.  That is, we have:

\begin{thm}
The hyperfinite II$_{1}$ factor $\mathcal{R}$ is the Fra\"iss\'e limit of the class of full matrix algebras (regarded as finite factors).
\end{thm}

\subsection{Finite width algebras}
We now return to considering C*-algebras.  Throughout this section we consider finite-dimensional C*-algebras as being unital, so in particular all the embeddings we consider will be unital *-homomorphisms.  To progress beyond UHF algebras, we need to consider more general classes of finite-dimensional algebras than just matrix algebras.  With the obstacles encountered in Proposition \ref{prop:AllFDNotFraisse} in mind, we make the following definitions.

\begin{defn}
\begin{enumerate}
\item{
A finite-dimensional C*-algebra $A$ has \defined{width $n$} if $A$ can be written as a direct sum of exactly $n$ matrix algebras.
}
\item{
A trace $\tau$ on a finite-dimensional C*-algebra $A$ is \defined{interior} if, when $\tau$ is written as a convex combination of the unique traces on the matrix algebras which appear as direct summands of $A$, none of the coefficients are $0$.  The trace $\tau$ is \defined{rational} if all of these coefficients are rational.
}
\end{enumerate}
\end{defn}

\begin{lem}\label{lem:AmalgamatingFDIntoMN}
Let $A, B, C$ be finite-dimensional C*-algebras of width $n$, and let $\alpha, \beta, \gamma$ be rational interior traces on $A, B, C$, respectively.  Let $\Phi : A \to B$ and $\Psi : A \to C$ be trace-preserving embeddings.  Then there exists $N \in \mathbb{N}$ such that $B$ and $C$ can be amalgamated into $M_N(\mathbb{C})$ over $A$ by trace-preserving embeddings.
\end{lem}

\begin{proof}
Write $A = M_{h_1}(\mathbb{C}) \oplus \cdots \oplus M_{h_n}(\mathbb{C})$.  For each $i$, let $\alpha_i$ be the unique trace on $M_{h_i}(\mathbb{C})$, and let $a_i \in \mathbb{Q}$ be such that $\alpha = \sum_{i=1}^n a_i\alpha_i$.  Write $B = M_{l_1}(\mathbb{C}) \oplus \cdots \oplus M_{l_n}(\mathbb{C})$, and $C = M_{k_1}(\mathbb{C}) \oplus \cdots \oplus M_{k_n}(\mathbb{C})$, and denote the traces on $B$ and $C$ by $\beta = \sum_{i=1}^n b_i\beta_i$ and $\gamma = \sum_{i=1}^n c_i\gamma_i$, respectively.  For each $i, j \leq n$, let $t_{i, j}$ be the multiplicity with which $A_i$ is embedded by $\Phi$ in $B_j$; similarly, let $q_{i, j}$ be multiplicity with which $A_i$ is embedded by $\Psi$ in $C_j$.

A direct computation from the definition of $\Phi$ (respectively, $\Psi$) being trace-preserving shows that for all $1 \leq j \leq n$,
\begin{equation}\label{eq0}
\sum_{i=1}^n \frac{b_i}{l_i}t_{j, i} = \frac{a_j}{h_j} = \sum_{i=1}^n \frac{c_i}{k_i}q_{j, i}.
\end{equation}

We consider the conditions necessary to create a trace-preserving amalgamation of $B$ and $C$ into $M_N(\mathbb{C})$.  For each $1 \leq i \leq n$, let $s_i$ be the multiplicity with which $M_{l_i}(\mathbb{C})$ is embedded in $M_N(\mathbb{C})$ by this hypothetical embedding, and let $r_i$ be similarly the multiplicity of the embedding of $M_{k_i}(\mathbb{C})$.  We immediately see that we must have
\begin{equation}\label{eq1}
\sum_{i=1}^nl_is_i = N = \sum_{i=1}^n k_ir_i.
\end{equation}
For the traces $\beta$ and $\gamma$ to be preserved (with respect to the unique trace $\delta$ on $M_N(\mathbb{C})$), we must additionally have, for each $1 \leq j \leq n$,
\begin{equation}\label{eq2}
b_j\sum_{i=1}^nl_is_i = l_js_j,
\end{equation}
and
\begin{equation}\label{eq3}
c_j\sum_{i=1}^nk_ir_i = k_jr_j.
\end{equation}

Finally, we must make our amalgamation respect $\Phi$ and $\Psi$.  It is sufficient to ensure that each $M_{h_i}(\mathbb{C})$ from $A$ embeds into $M_N(\mathbb{C})$ via $B$ and $C$ with the same multiplicities.  That is, we must satisfy the following for all $1 \leq j \leq n$:
\begin{equation}\label{eq4}
\sum_{i=1}^n t_{j, i} s_i = \sum_{i=1}^n q_{j, i}r_i.
\end{equation}

Finding any positive integers $s_1, \ldots, s_n, r_1, \ldots, r_n$ satisfying \ref{eq1}, \ref{eq2}, \ref{eq3}, and \ref{eq4} will complete the proof.

If we view Equation \ref{eq2} as a linear system in variables $s_i$ then the facts that $\sum_{i=1}^nb_i = \sum_{i=1}^nc_i = 1$ and all $b_i, c_i \neq 0$ imply that the system of equations \ref{eq2} is equivalent to
\[s_i = \frac{b_il_n}{b_nl_i}s_n \qquad \text{ for all } i < n,\]
and similarly Equation \ref{eq3} is equivalent to
\[r_i = \frac{c_ik_n}{c_nk_i}r_n \qquad \text{ for all } i < n.\]
Given these conditions, Equation \ref{eq1} reduces to 
\[r_n = \frac{l_nc_n}{b_nk_n}s_n.\]
If we choose any $s_n$ and define the remaining $r_i, s_i$ as above, straightforward substitution shows that Equation \ref{eq4} follows from Equation \ref{eq0}.  Therefore if $s_n \in \mathbb{N}$ is chosen so that the above formulas for the $s_i, r_i$ all yield integer values, then Equations \ref{eq1} - \ref{eq4} will be satisfied.
\end{proof}

\begin{prop}\label{prop:FixedWidthAP}
The class of finite-dimensional algebras of width $n \geq 2$ with a distinguished interior trace has AP.
Moreover, we can always choose the amalgam to have a rational trace. 
\end{prop}

\begin{proof}
Let $A, B, C$ be algebras of width $n$ with distinguished traces $\alpha, \beta, \gamma$, and let $\Phi : A \to B$ and $\Psi : A \to C$ be morphisms which each preserve $\alpha$.  By continuity, and the fact that $\alpha, \beta, \gamma$ are interior, the maps $\Phi$ and $\Psi$ each preserve an open neighbourhood of traces around $\alpha$.  Let $U$ be the intersection of these neighbourhoods, so $\Phi$ and $\Psi$ both preserve $U$.

Let $\tau_1, \ldots, \tau_n$ be rational traces on $A$ which form the vertices of an $(n-1)$-simplex contained in $U$.  Apply Lemma \ref{lem:AmalgamatingFDIntoMN} to each $\tau_i$ to produce matrix algebras $M_{N_1}(\mathbb{C}), \ldots, M_{N_n}(\mathbb{C})$ which embed $B$ and $C$ over $A$ with trace-preserving embeddings.  Let $D = M_{N_1}(\mathbb{C}) \oplus \cdots \oplus M_{N_n}(\mathbb{C})$, and embed $B$ and $C$ into $D$ by taking the direct sum of the embeddings into each $M_{n_i}(\mathbb{C})$; let $\Theta$ be the resulting embedding of $A$ into $D$.  The extremal traces on $D$ are mapped by $\Theta$ to the $\tau_i$, so by convexity there is some interior rational trace $\delta$ on $D$ which is mapped by $\Theta$ to $\alpha$.  Then $(D, \delta)$ is the required amalgam of $B$ and $C$ over $A$.
\end{proof}

We can now show that certain classes of finite-dimensional algebras are Fra\"iss\'e classes.  To obtain information about the Fra\"iss\'e limits we will use the $K_0$ functor.  To each unital C*-algebra $A$ is associated an abelian group $K_0(A)$, and to each embedding $f : A \to B$ an injective group homomorphism $K_0(f) : K_0(A) \to K_0(B)$.  Since we will not explicitly need the construction of $K_0$, we refer the reader to \cite{Davidson1996}, \cite{RoLaLa:Introduction},  or \cite{Farah2013a} for the definition.

\begin{thm}\label{Theorem:FraisseLimitAF}
For each $n \geq 2$, and each interior trace $\tau$ on $\mathbb{C}^n$, the class $\mc{K}(n, \tau)$ of finite-dimensional C*-algebras $A$ of width $n$ with a distinguished interior trace $\alpha$ such that there is an embedding of $\mathbb{C}^n$ into $A$ which preserves $\tau$, is a Fra\"iss\'e class.

The Frai\"ss\'e limit of $\mc{K}(n, \tau)$ is simple, has a unique trace, and is not self-absorbing.  As an abelian group, the $K_0$ group of the Fra\"iss\'e limit is divisible and of rank $n$.  Hence when $n \neq m$, the limits of $\mc{K}(n, \tau)$ and $\mc{K}(m, \sigma)$ are non-isomorphic.
\end{thm}
\begin{proof}
It follows from Proposition \ref{prop:FixedWidthAP} that this class has AP, and since this class has a minimal element, JEP is a consequence of AP. 
By  Proposition~\ref{prop:FixedWidthAP} we have countably members of $\mc{K}(n, \tau)$ (namely, finite-dimensional 
algebras with distinguished rational traces) such that every other member of $\mc{K}(n, \tau)$  embeds into one of them. 
Since the space of substructures of a fixed member of $\mc{K}(n, \tau)$ is separable in  $d^{\mc{K}}$, 
we conclude that $\mc{K}(n, \tau)$ has WPP. 

Let $A$ denote the Fra\"iss\'e limit of $\mc{K}(n, \tau)$.  It is clear from the proof of Proposition \ref{prop:FixedWidthAP} that whenever a finite-dimensional algebra $B$ appears in the construction of $A$, at some future stage there is a finite-dimensional algebra $C$ such that each direct summand of $B$ embeds into each direct summand of $C$.  By \cite[Corollary III.4.3]{Davidson1996} the limit $A$ is simple.

At each stage of the amalgamation in the proof of Proposition \ref{prop:FixedWidthAP} we have a $(B, \rho) \in \mc{K}(n, \tau)$, and we choose an open set around $\rho$ which is preserved by the relevant embeddings.  Given any trace $\sigma$ on $B$ other than $\rho$, in a future stage we may amalgamate with $(B, \rho)$ again, this time choosing an open set around $\rho$ which does not include $\sigma$.  So only the trace $\rho$ is preserved to the limit algebra $A$, and hence $A$ has a unique trace.

For the remaining claims, we consider $K_0(A)$.  For any choice of sequence $A_k$ from $\mc{K}(n, \tau)$ such that $A = \overline{\bigcup_{k\geq 1}A_k}$, we have $K_0(A) = \varinjlim K_0(A_k)$ (see \cite[Theorem IV.3.3]{Davidson1996}).  Each $A_k$ is a direct sum of exactly $n$ matrix algebras, so as abelian groups, $K_0(A_k) \cong \mathbb{Z}^n$.  The maps in $\mc{K}(n, \tau)$ are embeddings, and so the maps in the direct limit of $K_0$ groups are injective.  For torsion-free groups rank can be defined directly in terms of linear independence, and it follows that the direct limit of rank $n$ torsion-free abelian groups via injective maps has rank $n$; therefore we have $\op{rank}(K_0(A)) = n$.

Finally, we show that $A$ is not self-absorbing.  By the Kunneth formula for C*-algebras \cite{Schochet} there is an injective map $K_0(A) \otimes K_0(A) \to K_0(A \otimes A)$.  As $K_0(A)$ has rank $n$, we have that $K_0(A) \otimes K_0(A)$ has rank $n^2$, and hence cannot be injected into the rank $n$ group $K_0(A)$.  Therefore $K_0(A\otimes A) \not\cong K_0(A)$, and also $A \not\cong A \otimes A$.
\end{proof}

\section{The Jiang-Su algebra}\label{Section:JiangSu}
The Jiang-Su algebra $\mc{Z}$ was constructed by Jiang and Su in \cite{JiangSu}. 
This infinite-dimensional algebra is K-theoretically indistinguishable from the one-dimensional algebra $\mathbb C$. 
The tensorial absorption of the $\mc{Z}$ plays a central role in Elliott's classification program of nuclear C*-algebras  
 (see e.g.\ \cite{elliott_regularity_2008} and the introduction to \cite{sato2014nuclear}).  $\mc{Z}$ exhibits many of the properties of a Fra\"iss\'e limit.  In this section we show that $\mc{Z}$ is indeed a Fra\"iss\'e limit.  We begin with some basic definitions and properties.

\begin{defn}
Fix $p, q \in \mathbb{N}$.  The \defined{dimension drop algebra} $\mc{Z}_{p, q}$ is defined to be
(we identify $M_p(\bbC)\otimes M_q(\bbC)$ and $M_{pq}(\bbC)$)
\[
\mc{Z}_{p, q} = \set{f \in C\bigl([0, 1], M_{pq}(\mathbb{C})\bigr) : f(0) \in M_p(\mathbb{C}) \otimes 1_q \mbox{ and } f(1) \in 1_p \otimes M_q(\mathbb{C})},
\]
considered as a C*-algebra with the operations inherited from $C([0, 1], M_{pq}(\mathbb{C}))$.

A dimension drop algebra $\mc{Z}_{p, q}$ is \defined{prime} if $p$ and $q$ are co-prime. 
\end{defn}

Prime dimension drop algebras are projectionless (i.e., do not have projections other than 0 and 1). As an inductive limit of projectionless algebras, 
$\cZ$ is projectionless as well, and moreover its  $K_0$  coincides with $K_0$ of $\bbC$.

 Given a probability measure $\mu$ on $[0,1]$ there is a natural trace $\tau_\mu$ on $\mc{Z}_{p,q}$ given by
$$
\tau_\mu(f) = \int^1_0 \tau(f(t)) d\mu
$$ 
where $\tau$ is the unique trace on $M_{pq}(\mathbb{C})$.  
By using Riesz representation theorem for bounded linear functionals on $C([0,1])$ 
and the uniqueness of traces on fibres of $\cZ_{p,q}$ one shows that all traces of $\mc{Z}_{p,q}$ are of this form, 
hence $T(\mc{Z}_{p,q})$ is affinely homeomorphic to the space of probability measures on $[0,1]$. 

We need to remind the reader of a number of facts about measures before we can define the class $\mc{K}$ for which $\mc{Z}$ is a Fra\"iss\'e limit.  We say that a probability measure $\mu$ on $[0,1]$ is 
faithful and diffuse if the function $u(t) =\mu([0,t])$ is a strictly increasing and continuous.  This will imply that the trace defined above as $\tau_\mu$ is faithful and $\mu$ is diffuse as a measure i.e. for every $F \subseteq [0,1]$ with $\mu(F) > 0$ there is $E \subset F$ such that $\mu(E) < \mu(F)$.

\begin{fact}\label{fact:measures}
If $\mu$ is a faithful and diffuse probability measure on $[0,1]$ and $u(t) = \mu([0,t])$ then 
 for any $f \in C([0,1])$, 
\[
\int^1_0 f d\mu = \int^1_0 f(u(t)) dt
\]
where $dt$ is Lebesgue measure on $[0,1]$.
\end{fact}

 We will say that a trace $\tau_\mu$ on $\mc{Z}_{p,q}$ is faithful and diffuse if the associated measure is.

\begin{fact}\label{fact:diffuse}
Suppose that $\tau_\mu$ and $\tau_\lambda$ are two faithful and diffuse traces on a prime dimension drop algebra $\mc{Z}_{p,q}$ then there is an automorphism $\sigma$ of $\mc{Z}_{p,q}$ such that $\tau_\mu = \tau_\lambda \circ \sigma$.
\end{fact}

\begin{proof} 
It suffices to prove this when $\lambda$ is Lebesgue measure on $[0,1]$.  Let $u(t) = \mu([0,t])$.  Then the map $\sigma:\mc{Z}_{p,q} \rightarrow \mc{Z}_{p,q}$ given by
$\sigma(f) = f(u)$ is easily seen to be the desired automorphism.
\end{proof} 

The class $\mc{K}$ that we will consider is the class of all pairs $(\mc{Z}_{p,q},\tau)$ where $p$ and $q$ are co-prime and $\tau$ is a faithful and diffuse trace on $\mc{Z}_{p,q}$.  The language for this class will contain the usual language of C*-algebras together with a relation for a trace.

The original construction of the Jiang-Su algebra was as an inductive limit of a sequence of prime dimension drop algebras. It has a unique (definable) trace which when we refer to it, we will call $\tau$.  When we consider $\mc{Z}$ as a structure in our language with a relation for the trace, we will mean that $\mc{Z}$ is expanded by this unique trace. The key properties of $\mc{Z}$ that we will need are contained in the following lemmas.  

\begin{lem}
\label{lem:JiangSuProperties}
Every $(A,\tau) \in \mc{K}$ embeds in a trace-preserving manner into $\mc{Z}$.  In fact, $\mc{Z}$ is an inductive limit of a chain $(A_n,\tau_n)$ from $\mc{K}$ where $(A_0,\tau_0) = (A,\tau)$.  In particular, $\mc{Z}$ is a $\mc{K}$-structure.
\end{lem}

\begin{proof} These facts follow immediately from the main construction in \cite{JiangSu}; see Propositions 2.5 and 2.8  \end{proof}

The following result is implicit in \S 3 of \cite{JiangSu}; we give a proof for completeness.
\begin{lem}
\label{lem:KhasJEP}
$\mc{K}$ has the joint embedding property.
\end{lem}

\begin{proof}
Suppose $(p,q) = 1$.  We will show that $A = \mc{Z}_{p,q}$ embeds into $B = \mc{Z}_{pq,k}$ for any prime $k > pq$.  Because of this inequality, we can write $k = ap + bq$ for some positve $a$ and $b$.  Define a *-homomorphism $\tilde \varphi: \mc{Z}_{p,q} \rightarrow C([0,1],M_{pqk}(\bbC))$ as follows, for $t \in [0,1]$:
\[
\tilde \varphi(f)(t) = \left (
\begin{array}{cccccc}
f(0) &\ldots  & 0 &  & &\\
 \vdots& \ddots &\vdots  & && \\
 0& \ldots& f(0) &&&\\
 &  &  & f(t) & \ldots&0  \\
& & & \vdots& \ddots & \vdots\\
& & & 0&\ldots  & f(t)
\end{array} \right )
\]
where there are $ap$ copies of $f(0)$ and $bq$ copies of $f(t)$ on the diagonal.  $\tilde \varphi(f)(0) = f(0) \otimes id$ and we can find a unitary $u(1)$ such that $u^*(1)\tilde\varphi(f)u(1) \in id \otimes M_k(\bbC)$.  If we choose a continuous path of unitaries $u$ on $[0,1]$ from $id$ to $u(1)$ then $\varphi(f) = u^*\tilde\varphi(f) u$ is our desired map.  We now want to see that $\varphi$ can be chosen to be trace-preserving in our class $\mc{K}$.  In light of Fact \ref{fact:diffuse}, if $\tau$ is the trace induced on $B$ by Lebesgue measure, we need to show that $\tau$ restricted to the image of $A$ under $\varphi$ is a faithful and diffuse trace on $A$.  But from the form of $\tilde \varphi$, this is clear.

Finally, suppose $(p,q) = 1$ and $(p',q') = 1$.  Let $n$ be a common multiple of $pq$ and $p'q'$ and $k$ some prime bigger than $n$.  Then from above, both $\mc{Z}_{p,q}$ and $\mc{Z}_{p',q'}$ embed into $\mc{Z}_{n,k}$ preserving any faithful and diffuse trace. 
\end{proof}

The following lemma will be critical for establishing that $\mc{K}$ has the near amalgamation property.  Here, if $u$ is a unitary, $\mathrm{Ad}\left( u\right) :x\mapsto uxu^{\ast }$ denotes the inner automorphism associated with $u$.
\begin{lem}
\label{lem:ZRoberts}
Suppose that $A \in \mc{K}$ and $\varphi ,\psi :A\rightarrow 
\mc{Z}$ are trace-preserving embeddings.  If $\bar{a}\in A$ and $\epsilon >0$,
then there is a unitary $u \in \mc{Z}$ such that 
\[
\| (\mathrm{Ad}(u) \circ \varphi)(\bar a) - \psi(\bar a) \| < \epsilon
\]
\end{lem}
\begin{proof}
This is an immediate consequence of 
Robert's  \cite[Theorem 1.0.1]{robert_classification_2012} once we make some observations.  Theorem 1.0.1 proves a result about algebras that are noncommutative CW (NCCW) complexes and ones which have stable rank one.
Dimension-drop algebras are examples of NCCW complexes. In order for an algebra to be stable rank one, the invertible elements of that algebra must be dense. We shall check the assumptions of Robert's  theorem hold for $\cZ$. 
Every invertible  in $\mathcal{Z}_{p,q}$ is a continuous function from $[0,1]$ into  the set of invertible elements of 
$M_{pq}(\mathbb C)$. Since invertible elements are dense in $M_{pq}(\bbC)$, 
it is an exercise in topology of $[0,1]$ to show  that the invertible elements are dense in $\mathcal{Z}_{p,q}$. 
Since every unitary in $M_{pq}(\bbC)$ is of the form $\exp(i a)$ for a self-adjoint $a$, 
a similar exercise shows  that every unitary in $\cZ_{p,q}$ is of the form $\exp(i a)$ for a self-adjoint $a$ and 
in particular that the unitary group of $\mathcal{Z}_{p,q}$ is connected.  
This shows that the group $K_1$ of $\cZ_{p,q}$  is trivial (this is the only fact about $K_1$ that we will need; we refer the reader to \cite{RoLaLa:Introduction} for more information). 
In particular $\cZ$ is an inductive limit of NCCW complexes with trivial $K_1$ and is stable rank one. 

Since prime dimension drop algebras are projectionless so is their inductive limit,~$\cZ$. 
Additionally, $\cZ$ has a unique trace $\tau$ and two positive elements $a$ and $b$ in $\cZ$ are approximately unitarily equivalent 
if and only if $\tau(a^n)=\tau(b^n)$ for all $n$. 

A very special case of Robert's theorem \cite[Theorem 1.0.1]{robert_classification_2012} implies 
that if~$A$ has stable rank 1 and 
$B$ is an inductive limit of NCCW complexes with trivial $K_1$,  unique trace, and the  above property of $\cZ$, 
then the following hold (see \S\ref{Section:AF} for the definition of $\Phi_*$). 
\begin{enumerate}
\item For every trace $\sigma$ of $A$ there is a unital *-homomorphism $\phi\colon A\to B$ such that 
$\phi_*(\tau)=\sigma$. 
\item Two homomorphism $\varphi,\psi\colon  A\to B$ 
are approximately unitarily equivalent if and only if $\varphi_*(\tau)=\psi_*(\tau)$. 
\end{enumerate}
The lemma now follows.
\end{proof}

We can now prove the main result.

\begin{thm}
\label{thm:JiangSuFraisse} The Jiang-Su algebra $\mathcal{Z}$ with its
distinguished trace is the Fra\"iss\'e limit of the Fra\"iss\'e class $\mc{K}$.
\end{thm}

\begin{proof}
This is automatic by Lemma \ref{lem:ZRoberts} if we can see that $\mc{K}$ is a Fra\"iss\'e class.  Lemma \ref{lem:KhasJEP} directly shows that $\mc{K}$ has the joint embedding property.  Since every element of $\cK$ embeds into $\mc{Z}$, $\mc{K}$ has the weak Polish property.  We are left to show that $\mc{K}$ satisfies the near amalgamation property.  Towards this end, suppose that $A, B$ and $C$ are in $\mc{K}$ and that 
$\varphi:A \rightarrow B$ and $\psi:A \rightarrow C$.  By Lemma \ref{lem:KhasJEP}, we can choose $D \in \mc{K}$ and maps $\varphi':B \rightarrow D$ and $\psi':C \rightarrow D$.  Now by Lemma \ref{lem:JiangSuProperties}, we can assume that $\mc{Z}$ is an inductive limit of $D_n$ from $\mc{K}$ such that $D_0 = D$.  So resetting the notation, we have maps $\varphi,\psi$ from $A$ into $D$ and $D$ begins an inductive chain $\langle D_n : n \in \bbN \rangle$ leading to $\mc{Z}$.  By Lemma \ref{lem:ZRoberts}, for a fixed $\epsilon > 0$ there is a unitary $u \in Z$ such that
\[
\| (\mathrm{Ad}(u) \circ \varphi)(\bar a) - \psi(\bar a) \| < \epsilon/3
\]
where $\bar a$ are generators for $A$.  By the definability of unitaries, there is some $n \in \bbN$ and some unitary $u' \in D_n$ so that $\|u - u'\| < \epsilon/3$.  $D_n$ will now work as the near amalgam of $\varphi$ and $\psi$.
\end{proof}

\begin{rmk}
Although this proof shows that the Jiang-Su algebra is a  Fra\"iss\'e limit, it is a bit unsatisfactory in that it uses the existence of the algebra itself to establish the key properties of the  Fra\"iss\'e class.  Additionally, it relies heavily on \cite{robert_classification_2012} in order to prove near amalgamation.  
In an earlier version of the present paper we asked whether there was a self-contained proof that K is a Fraisse class. Such a proof was found by Masumoto in \cite{Masumoto}.
\end{rmk}

\section{L\'{e}vy automorphism groups}

\label{Section:Levy}

A Polish group $G$ is \emph{extremely amenable }if every continuous action
of $G$ on a compact space has a fixed point (see~\cite{pestov_dynamics_2006}). 
Suppose that $\left(
H_{n},d_{n}\right) _{n\in \mathbb{N}}$ is a sequence of compact metric
groups equipped with their normalized Haar measures~$\mu_{H_n}$. 
 The sequence $\left( H_{n}\right)
_{n\in \mathbb{N}}$ has the \emph{L\'{e}vy concentration property} if for
any sequence $A_{n}\subset H_n$ of Borel subsets such that $\liminf_{n}\mu
_{H_{n}}(A_{n})>0$ and for every $\varepsilon >0$%
\begin{equation*}
\lim_{n\rightarrow \infty }\mu _{H_{n}}\left\{ x\in H_n:\exists a\in A_{n},d\left(
a,x\right) \leq \varepsilon \right\} =1\text{;}
\end{equation*}%
see also \cite[Definition 1.2.6 and Remark 1.2.9]{pestov_dynamics_2006}. 
A Polish group is \emph{L\'{e}vy} if it admits an increasing sequence $\left(H_{n}\right) _{n\in \mathbb{N}}$ of compact subgroups with dense union with
the L\'{e}vy concentration property with respect to the metrics induced by a
compatible metric on $G$. 
 Every L\'{e}vy group is extremely amenable \cite[Theorem 4.1.3]%
{pestov_dynamics_2006}.

If $M$ is a II$_{1}$ factor then the automorphism group $\mathrm{Aut}(M)$ of 
$M$ is a Polish group with respect to the topology of pointwise convergence
in trace-norm. Similarly if $A$ is a separable C*-algebra then the
automorphism group $\mathrm{Aut}(A)$ of $A$ is a Polish group with respect
to the topology of pointwise convergence in norm.

Let  $U_n$ denote the unitary group of $M_n(\bbC)$. It can be naturally identified with 
a subgroup of the unitary 
group of both the hyperfinite II$_1$ factor $\cR$, as well as of the unitary  group of 
  a UHF algebra whose supernatural number is divisible by $n$. 
  The metrics induced by these embeddings correspond to the trace-norm and to the operator norm, respectively. 
  The groups $\SU_n=\{u\in U_n: \det(u)=1\}$ form 
  a L\'evy sequence with respect to either metric (\cite[Theorem~4.1.14]{pestov_dynamics_2006}).
We note that all automorphisms of $M_n(\bbC)$ are inner and that 
 $\Aut(M_n(\bbC))$ is naturally isomorphic to $\SU_n$, via $u\mapsto \Ad u$.

 The proof of the first two parts of the following Proposition are well-known
(cf. \cite{giordano_extremely_2002}), but for the convenience of the reader we include outlines of their proofs, as well as a more detailed proof of the third claim.

\begin{prop}
\label{Proposition:Levy}The automorphism groups of

\begin{enumerate}
\item the hyperfinite II$_{1}$ factor,

\item UHF algebras, and

\item the AF algebras obtained in\ Theorem \ref{Theorem:FraisseLimitAF}
\end{enumerate}

are L\'{e}vy and, in particular, extremely amenable.
\end{prop}

\begin{proof} By the above, for $n\in \bbN$ the group  
$
H_n:=\{\Ad u:u\in \SU_{n}\}
$
can be identified with a compact subgroup of $\Aut(\cR)$. 
These groups have the L\'evy approximation property, and 
since    $\bigcup_{n\in \mathbb{N}} U_n$ is dense $U(\mathcal{R})$
and all automorphisms of $\cR$ are approximately inner, $\bigcup_n H_n$
is dense in $\Aut(\cR)$.  Therefore (1) follows. 

The proof of (2) is identical, although the groups $H_n$ are now considered with a different metric.

In order to prove (3), fix $m\geq 2$ and let  $A$ be one of the AF algebras constructed in
Theorem \ref{Theorem:FraisseLimitAF} with Bratteli diagram of width $m$. Since $K_{0}\left( A\right) $ is
linearly ordered, all automorphisms of $A$ are approximately inner by  Elliott's classification 
of AF algebras \cite[Theorem IV.4.3]{Davidson1996}. 
Writing $A$ as an inductive limit of finite-dimensional algebras $A_n$, 
we represent $U(A)$ as an inductive limit of $U(A_n)$ and 
$\Aut (A)$ as an inductive limit of $\Aut (A_n)$. 
The algebra $A_n$ is a direct sum of matrix algebras $M_{k(i)}(\bbC)$ for $1\leq i\leq m$
and therefore $\Aut(A_n)\cong \prod_{i\leq m}\SU_{k(i)}$. 
An inspection of the proof of Theorem~\ref{Theorem:FraisseLimitAF}
shows that $\lim_n \min_{i\leq m} k(i)=\infty$. It is now an easy exercise to show  that  
the sequence $\Aut(A_n)$ has the L\'evy property, and (3) follows. 
\end{proof}

\section{A Ramsey theorem for matrix algebras}

\label{Section:Ramsey}

In this section we deduce from Proposition \ref{Proposition:Levy}
Ramsey-type results for matrix algebras. We will use the correspondence
between extreme amenability of a Fra\"iss\'e limit and the Ramsey property
of the corresponding Fra\"iss\'e class established in \cite[Theorem 3.10]%
{Melleray2014} building on a previous results in the discrete case from \cite%
{kechris_fraisse_2005}.

Suppose that $\mathcal{K}$ is a Fra\"{\i}ss\'{e} class in the sense of
Definition \ref{Definition:Fraisse}. If $A,B$ are elements of $\mathcal{K}$
with distinguished set of generators $\bar{a}$ for $A$, denote by $^{A}B$
space of embeddings of $A$ inside $B$ endowed with the metric%
\begin{equation*}
\rho _{\bar{a}}\left( \varphi ,\psi \right) =\max_{i}d\left( \varphi
(a_{i}),\psi (a_{i})\right) \text{.}
\end{equation*}%
A \emph{coloring }of $^{A}B$ is a $1$-Lipschitz map $\gamma
:{}^{A}B\rightarrow \left[ 0,1\right] $.

Suppose that $\mathcal{K}$ satisfies the property that $^{A}B$ is compact for every $A, B \in \mathcal{K}$.
In this case, the definition of 
the \emph{approximate Ramsey property } (\cite{Melleray2014}, Def. 3.3) is equivalent to:  for
every $A,B\in \mathcal{K}$ and 
every $\varepsilon >0$, there is $%
C\in \mathcal{K}$ such that for any coloring $\gamma $ of $^{A}C$ there is $%
\beta \in {}^{B}C$  such that $\gamma(\beta \circ -)$ varies by at most $\varepsilon$ on~${}^{A}B$.

In \cite{Melleray2014}, a version of the following is proved as Proposition 3.4.
\begin{prop}\label{Mel3.4} Suppose that $\mathcal{K}$ is a Fra\"iss\'e class with limit $M$ and for all $A, B \in \mathcal{K}$, ${}^{A}B$ is compact then the  following are equivalent:
\begin{enumerate}
\item $\mathcal{K}$ has the approximate Ramsey property.
\item For every $A, B \in \mathcal{K}$, $\varepsilon >0$, and every coloring $\gamma$ of ${}^{A}M$, there is $\beta \in {}^{B}M$ such that 
$\gamma(\beta \circ -)$ varies by at most $\varepsilon$ on ${}^{A}B$; we say $M$ has the approximate Ramsey property.
\end{enumerate}
\end{prop}

The following result can be proved with
the same methods as \cite[Theorem 3.10]{Melleray2014}.

\begin{thm}
\label{Theorem: extremely amenable} Suppose that $M$ is the limit of a Fra\"iss\'e class $\mathcal{K}$. The following statements are equivalent:
\begin{enumerate}
\item $\mathrm{Aut}(M)$ is extremely amenable.
\item $\mathcal{K}$ has the approximate Ramsey property.
\end{enumerate}
\end{thm}

Suppose that $B$ is a unital subalgebra of the hyperfinite II$_1$ factor $\mathcal{R}$. Endow the space 
${}^{M_{k}(\mathbb{C})}B$ of unital embeddings of $M_{k}(\mathbb{C})$ into $B$ with the
metric%
\begin{equation*}
d_{2}\left( \alpha ,\alpha ^{\prime }\right) =\sup_{\left\Vert x\right\Vert
\leq 1}\left\Vert \left( \alpha -\alpha ^{\prime }\right) (x)\right\Vert _{2}%
\text{.}
\end{equation*}

The following is an immediate corollary of Proposition \ref{Mel3.4}, Theorem \ref{Theorem: extremely
amenable} and the extreme amenability of $\mathrm{Aut}(\mathcal{R})$.

\begin{thm}
\label{Theorem:Ramsey-trace-1}
The class of matrix algebras equipped with the metric $d_2$ and its Fra\"iss\'e limit, $\mathcal{R}$, have the approximate Ramsey property.
\end{thm}

Using the extreme amenability of the automorphism groups of infinite type
UHF algebras one can obtain similar results for matrix algebras with respect
to the operator norm. If $q=\prod_{p}p^{n_{p}}$ for $n_{p}\in \left\{
0,\infty \right\} $, then we denote by $\mathbb{M}_{q}$ the infinite type
UHF algebras with associated supernatural number $q$. For $A\subset \mathbb{M%
}_{q}$ define $^{M_{k}(\mathbb{C})}A$ to be the set of embeddings of $M_{k}\left( 
\mathbb{C}\right) $ into $A$ endowed with the metric%
\begin{equation*}
d\left( \alpha ,\alpha ^{\prime }\right) =\sup_{\left\Vert x\right\Vert \leq
1}\left\Vert \left( \alpha -\alpha ^{\prime }\right) \left( x\right)
\right\Vert \text{.}
\end{equation*}%

\begin{thm}
For any supernatural number $q$, both $\mathbb{M}_q$ and its associated Fra\"iss\'e class have the approximate Ramsey property.
\end{thm}

Finally one can use the fact that the algebra $\mathcal{K}(H) $
of compact operators is the Fra\"iss\'e limit of the class of full matrix
algebras, and that $\mathrm{\mathrm{Aut}}\left( \mathcal{K}(H)
\right) $ is extremely amenable to obtain the analogues of the above results
where one considers not necessarily unital injective *-homomorphisms as
embeddings.  The same results hold for the finite width AF algebras and their associated Fra\"iss\'e classes as described in section \ref{Section:AF}.

\section{Future work}

\label{Section:future}


Both the Jiang-Su algebra and the infinite type UHF algebras are examples of
 strongly self-absorbing
C*-algebras   (\cite{toms_strongly_2007}). 
A unital C*-algebra $D$ is \emph{strongly self-absorbing} if 
there is a sequence of unitaries $u_n$ in $D\otimes D$ such that
\[
\Phi(a)=\lim_n u_n(a\otimes 1) u_n^*
\]
is well-defined for all $a\in D$ and  $\Phi$ is an isomorphism between $D$ and $D\otimes D$. 

In addition to $\cZ$ and the infinite type UHF algebras, the  only other currently known examples of strongly self-absorbing algebras 
are the Cuntz algebras $\mathcal{O}_{2}$ and $\mathcal{O}_{\infty }$ together with
the tensor products with $\mathcal{O}_{\infty }$ and infinite type UHF
algebras. Strongly self-absorbing algebras play a pivotal role in Elliott's classification program for nuclear, simple, separable, 
unital C*-algebras (see \cite[Chapters 5 and 7]{rordam_classification_2002} for the role of 
  $\mathcal{O}_{2}$ and $\mathcal{O}_{\infty }$ and 
  the more recent \cite{elliott_regularity_2008} and \cite{sato2014nuclear} for the role of $\cZ$)
  These algebras also have remarkable model-theoretic properties (see \cite[\S 2.2 and \S 4.5]{Fa:Logic} and \cite{FaHaRoTi}).  
Every strongly self-absorbing C*-algebra is an atomic model of its theory, and all atomic models can be viewed as 
Fra\"{\i}ss\'{e} limits of their type space.
Nevertheless, strongly self-absorbing C*-algebras share a number of properties with the Fra\"\i ss\'e limits not common to all atomic models, 
and it is natural to conjecture that all known, and perhaps all, strongly self-absorbing algebras
can be construed as    
Fra\"{\i}ss\'{e} 
limits of  
Fra\"{\i}ss\'{e} 
classes from which information about their automorphism group may be extracted. 

\begin{problem}
\label{Problem:Fraisse?}
Let $A$ be a strongly self-absorbing C*-algebra. Is $A$ a nontrivial Fra\"{\i}ss\'{e} limit?
\end{problem}

Since all strongly self-absorbing algebras are singly generated, and  $\mathcal{O}_{2}$ is moreover 
the universal algebra with two generators satisfying particularly simple relations, 
it may be necessary to consider 
Fra\"{\i}ss\'{e}  categories other than
C*-algebras, such as (unital) operator spaces (see \cite{lupini_uniqueness_2014}).


The important first step in proving that a nuclear algebra $A$ is strongly self-absorbing is to prove that it is tensorially self-absorbing, i.e., that $A \otimes A \cong A$.  Proofs that $\cO_2$ and $\cZ$ enjoy this property are nontrivial, and Elliott's proof that $\cO_2 \otimes \cO_2 \cong \cO_2$ in particular precipitated remarkable progress (see \cite{rordam_classification_2002}, \cite{elliott_regularity_2008}). 

In the case of $\cZ$, we note that if one considers the class $\mc{K}$ of dimension-drop algebras with distinguished traces as used in Section 4, we could modify the construction by considering a new class $\mc{K}'$ which is just the closure of $\mc{K}$ under taking finite tensor products together with the induced traces.  It would be interesting to know if this class is a Fra\"{\i}ss\'{e} class.  If so, this would give a direct proof that $\cZ$ is self-absorbing.


It is possible that viewing other strongly self-absorbing algebras as Fra\"iss\'e limits  may result in new proofs of tensorial self-absorbtion.  Such proofs would give information about these algebras, and this technique may also be useful in understanding Jacelon's non-unital analogue of $\cZ$ (\cite{jacelon2012simple}). 

\begin{problem} Is Jacelon's simple, monotracial, stably projectionless C*-algebra~$\cW$ a nontrivial  Fra\"{\i}ss\'{e} limit? 
Is $\cW\otimes \cW\cong \cW$? 
\end{problem} 

The construction of $\cW$ resembles the construction of $\cZ$, with the role of dimension-drop algebras being played by 
the so-called Razak building blocks (\cite{razak2002}). 

Another goal of this research is to shed new light on the automorphisms groups of
strongly self-absorbing C*-algebras such as $\mathcal{Z}$, $\mathcal{O}_{2}$, and $\mathcal{%
O}_{\infty }$.
 For example an affirmative
answer to Problem \ref{Problem:Fraisse?} would be a first step towards the
solution of the following problem.

\begin{problem}
\label{Problem:ExtremelyAmenable?} Suppose $A$ is strongly self-absorbing. 
Is $\mbox{Aut}(A)$ extremely amenable?
\end{problem}

\begin{problem}[{\cite[Question 9.1]{sabok_completeness_2013}}]
\label{Problem:universal?}
Is $\mbox{Aut}(\mathcal{O}_{2})$ a universal Polish group?
\end{problem}

Note that \cite[Theorem~7.4]{FaToTo:Turbulence} and the main result of 
\cite{sabok_completeness_2013} together imply that $\Aut(\cO_2)$ induces the universal orbit equivalence 
relation for Polish group actions. Moreover by Kirchberg's $\mathcal{O}_{2}$-absorption theorem \cite{kirchberg_exact_1995} every simple, 
separable,  nuclear and unital C*-algebra $A$ satisfies $A\otimes \mathcal{O}_{2}\cong 
\mathcal{O}_{2}$. In particular the automorphism group of $A$ embeds into
the automorphism group of $\mathcal{O}_{2}$ via the map $\alpha \mapsto
\alpha \otimes id_{\mathcal{O}_{2}}$.  
\bibliographystyle{amsplain}
\bibliography{Fraisse}

\providecommand{\bysame}{\leavevmode\hbox to3em{\hrulefill}\thinspace}
\providecommand{\MR}{\relax\ifhmode\unskip\space\fi MR }
\providecommand{\MRhref}[2]{%
  \href{http://www.ams.org/mathscinet-getitem?mr=#1}{#2}
}
\providecommand{\href}[2]{#2}
\begin{thebibliography}{10}

\bibitem{BenYaacov2013}
I.~Ben~Yaacov, \emph{Fra\"iss\'e limits of metric structures}, Journal of
  Symbolic Logic \textbf{80} (2015), no.~1, 100--115.

\bibitem{BYBHU}
I.~Ben~Yaacov, A.~Berenstein, C.W. Henson, and A.~Usvyatsov, \emph{Model theory
  for metric structures}, Model Theory with Applications to Algebra and
  Analysis, Vol. II (Z.~Chatzidakis et~al., eds.), London Math. Soc. Lecture
  Notes Series, no. 350, Cambridge University Press, 2008, pp.~315--427.

\bibitem{blackadar_operator_2006}
B.~Blackadar, \emph{Operator algebras}, Encyclopaedia of Mathematical Sciences,
  vol. 122, Springer-Verlag, Berlin, 2006.

\bibitem{Davidson1996}
K.~Davidson, \emph{{$C^*$}-algebras by example}, American Mathematical Society,
  Providence, {RI}, 1996.

\bibitem{EagleFarahKirchbergVignati}
C.~J. Eagle, I.~Farah, E.~Kirchberg, and A.~Vignati, \emph{Quantifier
  elimination in {C*}-algebras}, {arXiv}:1502.00573, 2015.

\bibitem{EagleGoldbringVignati}
C.~J. Eagle, I.~Goldbring, and A.~Vignati, \emph{The pseudoarc is a
  co-existentially closed continuum}, {arXiv}:1503.03443, 2015.

\bibitem{EagleVignati}
C.~J.\ Eagle and A.~Vignati, \emph{Saturation and elementary equivalence of
  {C*}-algebras}, Journal of Functional Analysis \textbf{269} (2015),
  2631--2664.

\bibitem{elliott_regularity_2008}
G.~A. Elliott and A.~S. Toms, \emph{Regularity properties in the classification
  program for separable amenable {$C^*$}-algebras}, Bulletin of the American
  Mathematical Society \textbf{45} (2008), no.~2, 229--245.

\bibitem{Fa:Logic}
I.~Farah, \emph{Logic and operator algebras}, Proceedings of the Seoul ICM
  (S.~Y. Jang, Y.~R. Kim, D.-W. Lee, and I.~Yie, eds.), vol.~II, Kyung Moon SA,
  2014, pp.~15--40.

\bibitem{Farah2013a}
\bysame, \emph{Selected applications of logic to classification problem of
  {C}*-algebras}, E-recursion, forcing and {C*}-algebras (C.T. Chong et~al.,
  eds.), Lecture Note Series, Institute for Mathematical Sciences, National
  University of Singapore, vol.~27, World Scientific, 2014, pp.~1--82.

\bibitem{FaHaRoTi}
I.~Farah, B.~Hart, M.~R{\o}rdam, and A.~Tikuisis, \emph{Relative commutants of
  strongly self-absorbing {C*}-algebras}, arXiv:1502.05228, 2015.

\bibitem{Farah2014a}
I.~Farah, B.~Hart, and D.~Sherman, \emph{Model theory of operator algebras
  {II}: Model theory}, Israel J. Math. \textbf{201} (2014), 477--505.

\bibitem{FaToTo:Turbulence}
I.~Farah, A.~S. Toms, and A.~T\"ornquist, \emph{Turbulence, orbit equivalence,
  and the classification of nuclear {C*}-algebras}, J. Reine Angew. Math.
  \textbf{688} (2014), 101--146.

\bibitem{fraisse_sur_1954}
R.~Fra{\"\i}ss{\'e}, \emph{Sur l'extension aux relations de quelques
  propri{\'e}t{\'e}s des ordres}, Annales Scientifiques de l'{\'E}cole Normale
  Sup{\'e}rieure. Troisi{\`e}me S{\'e}rie \textbf{71} (1954), 363--388.

\bibitem{giordano_extremely_2002}
T.~Giordano and V.~Pestov, \emph{Some extremely amenable groups}, Comptes
  Rendus Mathematique \textbf{334} (2002), no.~4, 273--278.

\bibitem{glimm_certain_1960}
J.~G. Glimm, \emph{On a certain class of operator algebras}, Transactions of
  the American Mathematical Society \textbf{95} (1960), no.~2, 318--340.

\bibitem{gromov_filling_1983}
M.~Gromov, \emph{Filling {R}iemannian manifolds}, Journal of Differential
  Geometry \textbf{18} (1983), no.~1, 1--147.

\bibitem{HartGoldbringSinclair}
B.~Hart, I.~Goldbring, and T.~Sinclair, \emph{The theory of tracial von
  {N}eumann algebras does not have a model companion}, J. Symb. Logic
  \textbf{78} (2013), no.~3, 1000--1004.

\bibitem{jacelon2012simple}
B.~Jacelon, \emph{A simple, monotracial, stably projectionless {C*}-algebra},
  Journal of the London Mathematical Society \textbf{87} (2013), 365--383.

\bibitem{JiangSu}
X.~Jiang and H.~Su, \emph{On a simple unital projectionless {$C^*$}-algebra},
  American Journal of Mathematics \textbf{121} (1999), 359--413.

\bibitem{kechris_fraisse_2005}
A.~S. Kechris, V.~Pestov, and S.~Todorcevic, \emph{Fra{\"\i}ss{\'e} limits,
  {R}amsey theory, and topological dynamics of automorphism groups}, Geometric
  \& Functional Analysis {GAFA} \textbf{15} (2005), no.~1, 106--189.

\bibitem{kirchberg_exact_1995}
E.~Kirchberg, \emph{Exact {C}*-algebras, tensor products, and the
  classification of purely infinite algebras}, Proceedings of the International
  Congress of Mathematicians, Vol. 1, 2 (Z{\"u}rich, 1994), Birkh{\"a}user,
  Basel, 1995, pp.~943--954.

\bibitem{kubis_proof_2013}
W.~Kubi{\'s} and S.~Solecki, \emph{A proof of uniqueness of the {G}urari\u\i\
  space}, Israel Journal of Mathematics \textbf{195} (2013), no.~1, 449--456.

\bibitem{lupini_uniqueness_2014}
M.~Lupini, \emph{Uniqueness, universality, and homogeneity of the
  noncommutative {G}urarij space}, arXiv:1410.3345, 2014.

\bibitem{Masumoto}
S.~Masumoto, \emph{Jiang-{S}u algebra as a {F}ra\"iss\'e limit},
  {arXiv}:1602.00124, 2016.

\bibitem{Melleray2014}
J.~Melleray and T.~Tsankov, \emph{Extremely amenable groups via continuous
  logic}, arxiv:1404.4590, 2014.

\bibitem{pestov_dynamics_2006}
V.~Pestov, \emph{Dynamics of infinite-dimensional groups}, University Lecture
  Series, vol.~40, American Mathematical Society, Providence, {RI}, 2006.

\bibitem{razak2002}
S.~Razak, \emph{On the classification of simple stably projectionless
  {$C^*$}-algebras}, Canad. J. Math. \textbf{54} (2002), no.~1, 138--224.

\bibitem{robert_classification_2012}
L.~Robert, \emph{Classification of inductive limits of 1-dimensional {NCCW}
  complexes}, Advances in Mathematics \textbf{231} (2012), no.~5, 2802--2836.

\bibitem{RoLaLa:Introduction}
M.~R{\o}rdam, F.~Larsen, and N.J. Laustsen, \emph{An introduction to {K}-theory
  for {C}$^*$ algebras}, London Mathematical Society Student Texts, no.~49,
  Cambridge University Press, 2000.

\bibitem{rordam_classification_2002}
M.~R{\o}rdam and E.~St{\o}rmer, \emph{Classification of nuclear {C}*-algebras.
  {E}ntropy in operator algebras}, Encyclopaedia of Mathematical Sciences, vol.
  126, Springer-Verlag, Berlin, 2002, Operator Algebras and Non-commutative
  Geometry, 7.

\bibitem{sabok_completeness_2013}
M.~Sabok, \emph{Completeness of the isomorphism problem for separable
  {C}*-algebras}, to appear in Invent. Math., 2015.

\bibitem{sato2014nuclear}
Y.~Sato, S.~White, and W.~Winter, \emph{Nuclear dimension and {$\mathcal
  Z$}-stability}, Invent. Math. \textbf{202} (2015), 893--921.

\bibitem{Schochet}
C.~Schochet, \emph{Topological methods for {C*}-algebras {II}: {G}eometric
  resolutions and the {K}\"unneth formula}, Pacific Journal of Mathematics
  \textbf{98} (1982), 399--445.

\bibitem{Schoretsanitis2007}
K.~Schoretsanitis, \emph{Fra\"iss\'e theory for metric structures}, Ph.D.
  thesis, University of Illinois at Urbana-Champaign, 2007.

\bibitem{Hannes2014}
H.~Thiel and W.~Winter, \emph{The generator problem for {$\cZ$}-stable
  {C*}-algebras}, Transactions of the American Mathematical Society
  \textbf{366} (2014), 2327--2343.

\bibitem{toms_strongly_2007}
A.~S. Toms and W.~Winter, \emph{Strongly self-absorbing {C}*-algebras},
  Transactions of the American Mathematical Society \textbf{359} (2007), no.~8,
  3999--4029.

\end{thebibliography}

\end{document}